\documentclass{article}
\usepackage{amssymb,amsmath,bm,geometry,graphics,graphicx,url,color}
\usepackage{amsfonts}
\usepackage{graphicx}
\usepackage{multirow}
\usepackage{amscd}
\usepackage{enumerate}
\usepackage{mathrsfs}
\usepackage{xypic}
\usepackage{stmaryrd}
\usepackage{fancybox}
\usepackage{exscale}
\usepackage{enumitem,array}%
\usepackage{graphicx}

\def\pdt2{\partial_t^2}
\def\pdx2{\partial_x^2}

\newcommand{\normmm}[1]{{\left\vert\kern-0.25ex\left\vert\kern-0.25ex\left\vert #1
    \right\vert\kern-0.25ex\right\vert\kern-0.25ex\right\vert}}

\newcommand{\abs}[1]{\left\vert#1\right\vert}

\def\RR{{\mathbb{R}}}

\def\ZZ{{\mathbb{Z}}}
\def\ii{\mathrm{i}}

\newtheorem{theo}{Theorem}[section]

\newtheorem{rem}[theo]{Remark}
\newtheorem{defi}[theo]{Definition}

\newtheorem{prop}[theo]{Proposition}

\def\no{\noindent}

\title{Long time conservations of two-step symmetric methods for charged-particle dynamics
in a  normal or strong magnetic field}

\author{Bin Wang\thanks{School of Mathematical
Sciences, Qufu Normal University, Qufu  273165, P.R.China;
Mathematisches Institut, University of T\"{u}bingen, Auf der
Morgenstelle 10, 72076 T\"{u}bingen, Germany. E-mail:~{\tt
wang@na.uni-tuebingen.de}} \and Xinyuan Wu\,\footnote{School of
Mathematical Sciences, Qufu Normal University, Qufu  273165,
P.R.China.  Department of Mathematics, Nanjing University, Nanjing
210093, P.R.China. E-mail:~{\tt xywu@nju.edu.cn}} \and Yajun
Wu\,\footnote{School of Mathematical Sciences, Qufu Normal
University, Qufu  273165, P.R.China. E-mail:~{\tt
1921170786@qq.com}} \and Ting Li\,\footnote{School of Mathematical
Sciences, Qufu Normal University, Qufu  273165, P.R.China.
E-mail:~{\tt 1009587520@qq.com}}   }

\begin{document}
\maketitle
\begin{abstract}
The study of the long time conservation for numerical methods poses
interesting and challenging questions from the
 point of view of geometric integration. In this paper, we analyze the
long time conservations of two-step symmetric methods for
charged-particle dynamics. Two two-step symmetric methods are
proposed and their long time behaviour is shown not only in a normal
magnetic field but also in a strong magnetic field. The approaches
to dealing with these two cases are based on the backward error
analysis and modulated Fourier expansion, respectively. It is
obtained from the analysis that the first method has an additional
superiority which the Boris method does not possess, and the second
method has better long time conservations than  the variational
method which was researched recently in the literature.
\medskip

\no{Keywords:}  charged-particle dynamics, two-step symmetric
methods, backward error analysis, modulated Fourier expansion

\medskip
\no{MSC (2000):} 65L05, 65P10, 78A35, 78M25
\end{abstract}

\section{Introduction}\label{intro}
The numerical investigation for charged-particle dynamics have
received much attention in the last few decades (see, e.g.
\cite{Boris1970,Hairer2017-1,Hairer2017-2,He2015,He2017,Tao2016}).
In this paper, we analyze the long time conservations of two
two-step symmetric methods   for solving charged-particle dynamics
of the form (see \cite{Hairer2018})
\begin{equation}\label{charged-particle sts}
\begin{array}[c]{ll}
\ddot{x}=\dot{x} \times  \frac{1}{\epsilon} B(x)+F(x), \quad
x(0)=x_0,\quad \dot{x}(0)=\dot{x}_0,\ \ t\in[0,T],
\end{array}
\end{equation}
where  $B(x) =  \nabla \times A(x)$ is a magnetic field with the
vector potential $A(x)\in \RR^3$, the position of a particle moving
in this field is denoted by $x(t)\in \RR^3$,  and $F(x) = -\nabla
U(x)$ is an electric field with the scalar potential $U(x)$.     The
energy of this dynamics
\begin{equation}\label{energy of cha}
E(x,v)=\frac{1}{2}\abs{v}^2+U(x)
\end{equation}
is preserved exactly along the solution $x$ and the velocity
$v=\dot{x}$ of the particle. It is assumed that the initial values
 are bounded as
\begin{equation}\label{ivb}x_0=\mathcal{O}(1),\ \
v_0:=\dot{x}_0=\mathcal{O}(1).
\end{equation}
In this work, we focus on the study of the following two regimes of
$\epsilon$:
\begin{itemize}
  \item one regime is that $\epsilon$ in \eqref{charged-particle sts} is assumed to be one which
  means that the magnetic field is ``normal";
    \item the other  regime is that $\epsilon$ in \eqref{charged-particle sts} is assumed to satisfy $0<\epsilon\ll 1$ which
  means that the magnetic field is ``strong".
\end{itemize}
For the normal magnetic field,
 if the scalar and vector potentials have the invariance properties
\begin{equation*}%\label{nvariance pro}
U(e^{\tau S}x)=U(x),\ \ \ e^{-\tau S} A(e^{\tau S}x)=A(x)
\end{equation*}
for all real $\tau$ with a skew-symmetric matrix $S$, then   the
momentum
\begin{equation}\label{momentum}
M(x,v)=(v+A(x))^{\intercal}Sx
\end{equation}
is conserved along solutions of \eqref{charged-particle sts}. For
the strong magnetic field, it has been shown in
 \cite{Hairer2018} that the magnetic moment
$$I(x,v)=\frac{1}{2}\frac{\abs{v_{\perp}}^2}{\abs{B(x) }}$$ is nearly
conserved along the solution over long time scales,  where
$v_{\perp}$ is orthogonal to $B(x)$ and is given by
$v_{\perp}=\frac{v \times B(x)}{\abs{B(x) }}$.

In order to effectively solve  charged-particle dynamics,  many
methods have been developed and studied in recent decades, such as
the Boris method (see, e.g. \cite{Boris1970,Hairer2017-1,Qin2013}),
symplectic or K-symplectic algorithms (see, e.g.
\cite{He2017,Tao2016,Webb2014,wang2018-new2,Zhang2016}), symmetric
multistep methods (see, e.g.
 \cite{Hairer2017-2}), volume-preserving
algorithms (see, e.g. \cite{He2015}), variational integrators in a
strong magnetic field  (see, e.g. \cite{Hairer2018}) and
  exponential integrators in a constant magnetic field (see, e.g. \cite{wang2018-new2}).

On the basis of these studies, this paper analyzes the long time
conservations of two two-step  symmetric methods for
charged-particle dynamics not only in a normal magnetic field but
also in a strong magnetic field. We will show the energy and
momentum conservations of the methods in a normal magnetic field and
derive the modified  energy and modified magnetic moment
preservations in a strong magnetic field. To this end, the backward
error analysis is employed for the first case, and the modulated
Fourier expansion is used for the strong magnetic field.

The  main  contributions of this paper are to show that the first
method has an additional superiority which the Boris method does not
have and that the second method has better long time conservations
than the variational method which was researched recently in
\cite{Hairer2018}.
 The paper is    organized as follows.
In Section \ref{sec:method}, we  formulate the schemes of methods
and discuss their basic properties.  Section \ref{sec:experiment}
gives the main results of this paper and carries out some
illustrative numerical experiments. Then the special result of the
first method is shown in Section \ref{sec:speci}, the results in a
normal magnetic field are proved in Section \ref{sec:normal}, and
the conservations in a strong magnetic field are shown  in Section
\ref{sec:strong}.  The last section is concerned with the
conclusions of this paper.

\section{Two methods and basic properties}\label{sec:method}
\subsection{The first method} Applying the implicit midpoint rule to the
charged-particle dynamics \eqref{charged-particle sts} and
approximating $x(t)$ as the linear interpolant of $x_n$ and
$x_{n+1}$ yields  the first method as follows.
\begin{defi}
\label{def: scheme 1}  The   first  method for solving  the
charged-particle dynamics \eqref{charged-particle sts}  is defined
by:
\begin{equation}\label{TSM1}
\left\{\begin{array}[c]{ll}x_{n+1}=x_{n}+
hv_{n+1/2},\\
v_{n+1}=v_{n}+h v_{n+1/2} \times \frac{1}{\epsilon}B(x_{n+1/2})+h
 F(x_{n+1/2}),
\end{array}\right.
\end{equation}
where $h$ is a stepsize, $v_{n+1/2}=\frac{v_{n}+v_{n+1}}{2}$ and
$x_{n+1/2}=\frac{x_{n}+x_{n+1}}{2}$. This method can be rewritten as
a two-step method
\begin{equation}\label{TSM1-twostep}
\begin{array}[c]{ll}x_{n+1}-2x_{n}+x_{n-1}=&\frac{h}{2}\big(x_{n+1}-x_{n}\big)
\times \frac{1}{\epsilon}B(x_{n+1/2})+
\frac{h}{2}\big(x_{n}-x_{n-1}\big) \times
\frac{1}{\epsilon}B(x_{n-1/2})\\
&+\frac{h^2}{2}\big(F(x_{n+1/2})+F(x_{n-1/2})\big).
\end{array}
\end{equation}
We denote this method by TSM1.
\end{defi}

\begin{prop}\label{vp thm}
Let $q=(x^{\intercal},v^{\intercal})^{\intercal}$ and
$G(q)=(v^{\intercal}, (v\times \frac{1}{\epsilon}
B(x)+F(x))^{\intercal})^{\intercal}$. The first method \eqref{TSM1}
is volume preserving for vector fields
$$\mathcal{D}=\{\textmd{vector fields satisfying}\  \det\big(I+\frac{h}{2}G'(q)\big)=\det\big(I-\frac{h}{2}G'(q)\big)\ \textmd{for all}\ h>0\
\textmd{and all}\ x\}.$$
\end{prop}
\begin{proof}
Denote the first method \eqref{TSM1}  by a map $\varphi_h.$ Then we
have $\varphi_h(q)=q+hG\big(\frac{\varphi_h(q)+q}{2}\big),$ which
leads to
$\varphi'_h(q)=I+\frac{h}{2}G'\big(\frac{\varphi_h(q)+q}{2}\big)(I+\varphi'_h(q)).$
Therefore, one gets   the condition for volume preservation
$$\det(\varphi'_h(q))=\frac{\det\Big(I+\frac{h}{2}G'\big(\frac{\varphi_h(q)+q}{2}\big)\Big)}{\det\Big(I-\frac{h}{2}G'\big(\frac{\varphi_h(q)+q}{2}\big)\Big)}
=1.$$

\end{proof}
\subsection{The second method}
A related variational integrator, which coincides with the first
  method for constant $B$, is
constructed as follows. We first  approximate $x(t)$ as the linear
interpolant of $x_n$ and $x_{n+1}$ and  then approximate the
integral of the Lagrangian $L(x(t), \dot{x}(t))$  by   the midpoint
rule (see, e.g., Example 6.3, Chap. VI
  of \cite{Hairer2006} and \cite{Webb2014}). This gives the following
variational integrator.

\begin{defi}
\label{def: scheme 2}  The second method  for solving  the
charged-particle dynamics \eqref{charged-particle sts} is defined by
a two-step method
\begin{equation}\label{TSM2}
\begin{array}[c]{ll}x_{n+1}-2x_{n}&+x_{n-1}
=\frac{h}{2}\frac{1}{\epsilon}A'^{\intercal}(x_{n+1/2})\big(x_{n+1}-x_{n}\big)
+ \frac{h}{2}\frac{1}{\epsilon}A'^{\intercal}(x_{n-1/2})
\big(x_{n}-x_{n-1}\big)  \\
&-h\frac{1}{\epsilon}\big(A(x_{n+1/2})-A(x_{n-1/2})\big)+\frac{h^2}{2}\big(F(x_{n+1/2})+F(x_{n-1/2})\big),
\end{array}
\end{equation}
and $v_{n+1}=2\frac{x_{n+1}-x_{n}}{h}-v_{n}.$ We denote this method
by TSM2.
\end{defi}

\begin{rem}
It is clear that these two methods are symmetric methods of order
two. From the fact that the second method \eqref{TSM2} is a
variational integrator, it follows that for \eqref{TSM2},  with the
momenta $p_n = v_n + A(x_n )$, the map $(x_n ,p_n
)\rightarrow(x_{n+1},p_{n+1} )$ is symplectic. It is noted that when
$B(x)$ is a constant magnetic field $B$ (which means that
$A(x)=-\frac{1}{2}x  \times B$), the method TSM2 becomes TSM1, and
has the following scheme
\begin{equation*}%\label{TSM1-twostep-cons}
\begin{array}[c]{ll}x_{n+1}-2x_{n}+x_{n-1}=&\frac{h}{2}\big(x_{n+1}-x_{n-1}\big)
\times \frac{1}{\epsilon}B
 +\frac{h^2}{2}\Big(F\big(x_{n+1/2}\big)+F\big(x_{n-1/2}\big)\Big).
\end{array}
\end{equation*}
 Thus the first method \eqref{TSM1} is also  symplectic once the magnetic field $B$ is constant.
\end{rem}

\section{Main results and numerical experiment}\label{sec:experiment}
\subsection{The energy preservation of TSM1 for quadratic scalar
potentials} For some special scalar potentials, the first method
TSM1 has an exact energy preservation stated below.
\begin{theo}\label{energy pre thm}
(\textbf{Energy conservation of TSM1 for quadratic scalar
potentials.}) If the scalar potential is quadratic,
$U(x)=\frac{1}{2}x^{\intercal} Q x+q^{\intercal}x$ with a symmetric
matrix $Q$, then  the first method \eqref{TSM1} preserves the energy
$E$ in \eqref{energy of cha} exactly for a normal or strong magnetic
field, i.e.,
$$E(x_{n+1},v_{n+1})=E(x_{n},v_{n})\qquad \textmd{for} \qquad n=0,1,\ldots.$$
\end{theo}

\subsection{Results in a normal magnetic field}
The results of this subsection are given for the methods applied to
charged-particle dynamics in a  normal   magnetic field.
\subsubsection{Results of TSM1}
\begin{theo}\label{nearly energy pre thm0}
(\textbf{Energy conservation of TSM1.}) If the numerical solution
stays in a compact set that is independent of $h$ and $B(x) = B$,
then the first method \eqref{TSM1} has a long time energy
conservation
$$\abs{E(x_{n+1/2},v_{n+1/2})-E(x_{1/2},v_{1/2})}\leq Ch^2\ \ \ \textmd{for}\ \  nh\leq h^{-N},$$
where $N\geq 2$ is an arbitrary truncation number and  the constant
  $C$ is independent of  $n$ and $h$ as long as
$nh\leq h^{-N}.$
\end{theo}

\begin{theo}\label{nearly mom pre thm0}
(\textbf{Momentum conservation of TSM1.}) Under the condition of
Theorem \ref{nearly energy pre thm0},  the first method \eqref{TSM1}
has a long time momentum conservation
$$\abs{M(x_{n+1/2},v_{n+1/2})-M(x_{1/2},v_{1/2})}\leq C_2h^2\ \ \ \textmd{for}\ \  nh\leq h^{-N},$$
where   $C_2$ is a generic constant which has the same meaning as
that of Theorem \ref{nearly energy pre thm0}.
\end{theo}
\begin{rem}
It follows from these two results  that the method TSM1 has similar
long time conservations of energy and momentum as Boris method whose
energy behaviour was researched in \cite{Hairer2017-1}. However, as
shown in Theorem \ref{energy pre thm}, if the scalar potential is
quadratic, TSM1 can preserve the energy exactly. This is an
additional superiority that the Boris method does not have.
\end{rem}

\subsubsection{Results of TSM2}
\begin{theo}\label{nearly energy pre thm02}
(\textbf{Energy conservation of TSM2.}) Assume that the numerical
solution stays in a compact set that is independent of $h$, then the
second method \eqref{TSM2} has a long time energy conservation
$$\abs{E(x_{n+1/2},v_{n+1/2})-E(x_{1/2},v_{1/2})}\leq Ch^2\ \ \ \textmd{for}\ \  nh\leq h^{-N},$$
where $N\geq 2$ is an arbitrary truncation number and  the constant
  $C$ is independent of  $n$ and $h$ as long as $nh\leq
h^{-N}.$
\end{theo}

\begin{theo}\label{nearly mom pre thm02}
(\textbf{Momentum conservation of TSM2.}) Under the condition of
Theorem \ref{nearly energy pre thm02},  the following momentum
conservation is true for the second integrator \eqref{TSM2}
$$\abs{M(x_{n+1/2},v_{n+1/2})-M(x_{1/2},v_{1/2})}\leq Ch^2\ \ \ \textmd{for}\ \  nh\leq h^{-N}.$$
\end{theo}

\begin{rem}
With these two results, it is confirmed that the method TSM2 has
good long time conservations of energy and momentum not only for a
constant magnetic field but also for a general magnetic field.
 This is a superiority of TSM2 over TSM1.
\end{rem}

\subsubsection{Experiment}\label{subse:expe}
In order to show the efficiency of the methods, we choose the
following two well-known methods for comparison.
\begin{itemize}
  \item BORIS. The  Boris method given in \cite{Boris1970}:
  $$x_{n+1}-2x_{n}+x_{n-1}=\frac{h}{2}\big(x_{n+1}-x_{n-1}\big)
\times \frac{1}{\epsilon}B(x_n)-h^2\nabla U(x_n).$$
  \item VARM. The variational method given by Example 6.2, Chap. VI
  of \cite{Hairer2006} and also discussed in \cite{Hairer2017-2,Webb2014}:
    $$x_{n+1}-2x_{n}+x_{n-1}=\frac{h}{2}\frac{1}{\epsilon}A'^{\intercal}(x_{n})\big(x_{n+1}-x_{n}\big)
-\frac{h}{2}\frac{1}{\epsilon}(A(x_{n+1})-A(x_{n})) -h^2\nabla
U(x_n).$$
\end{itemize}

We consider the charged particle system \eqref{charged-particle sts}
with (see \cite{Hairer2017-2})
$$\epsilon=1,\quad U(x)=\frac{1}{100\sqrt{x_1^2+x_2^2}},\quad B(x)=(0,0,\sqrt{x_1^2+x_2^2})^{\intercal}.$$
The momentum of this problem is given by
$$M(x,v)=(v_1-\frac{x_2}{3}\sqrt{x_1^2+x_2^2})x_2-(v_2+\frac{x_1}{3}\sqrt{x_1^2+x_2^2})x_1.$$
 We choose the initial
values $x(0)=(0,1,0.1)^{\intercal},\
v(0)=(0.09,0.05,0.20)^{\intercal} $ and  solve it in $[0,10000]$
with $h=0.1, 0.05$.  The conservations  of the energy  and momentum
for different methods  are displayed in Figures \ref{p1}-\ref{p2}.
It can be observed from Figures \ref{p1}-\ref{p2} that the methods
TSM1 and TSM2 proposed in this paper show remarkable long-time
numerical behaviour.

 \begin{figure}[ptb]
\centering
\includegraphics[width=14cm,height=3cm]{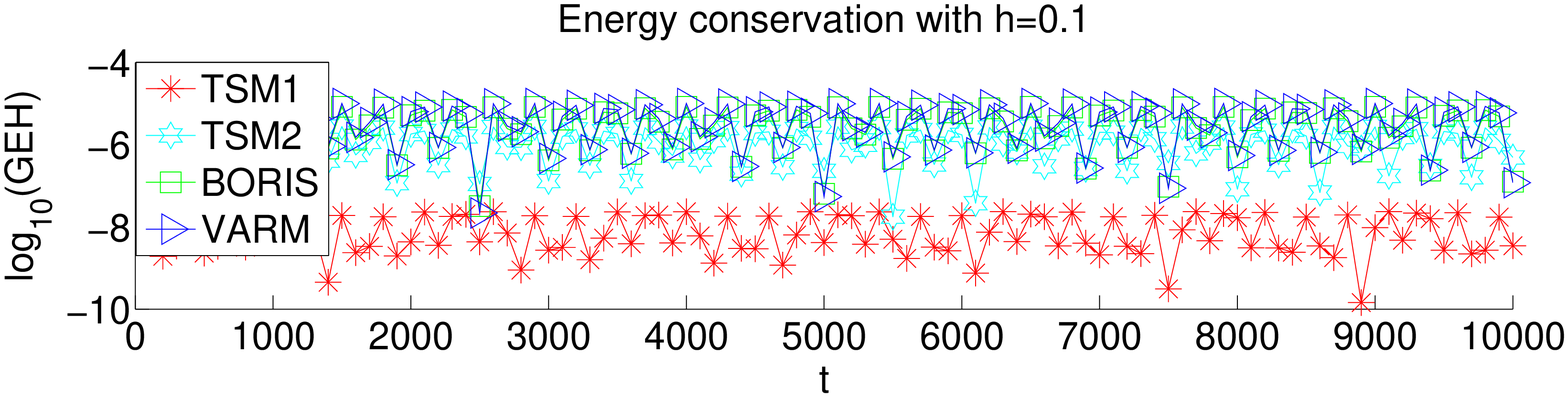}
\includegraphics[width=14cm,height=3cm]{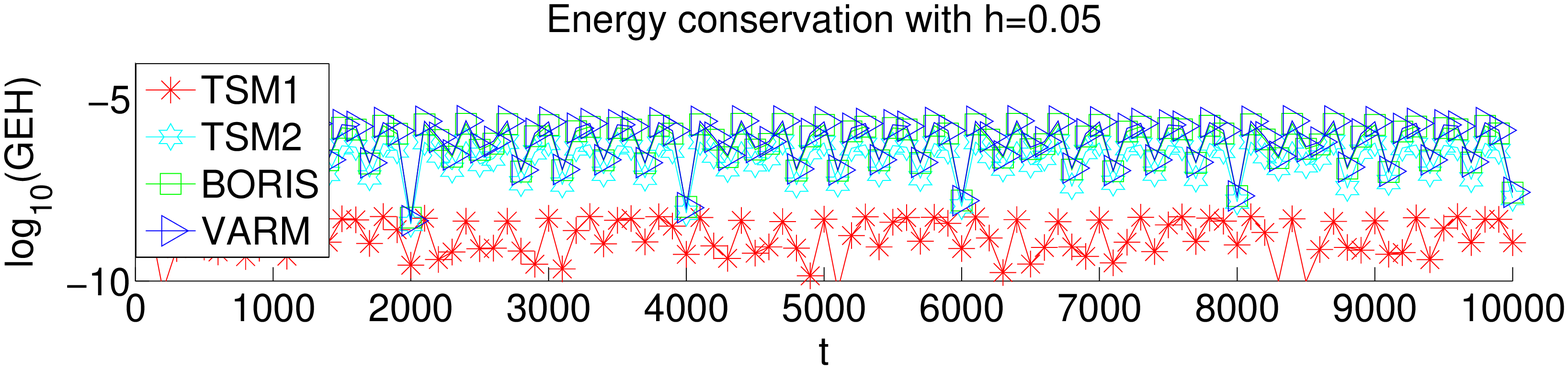}
\caption{The  errors of energy against  $t$.} \label{p1}
\end{figure}

\begin{figure}[ptb]
\centering
\includegraphics[width=14cm,height=3cm]{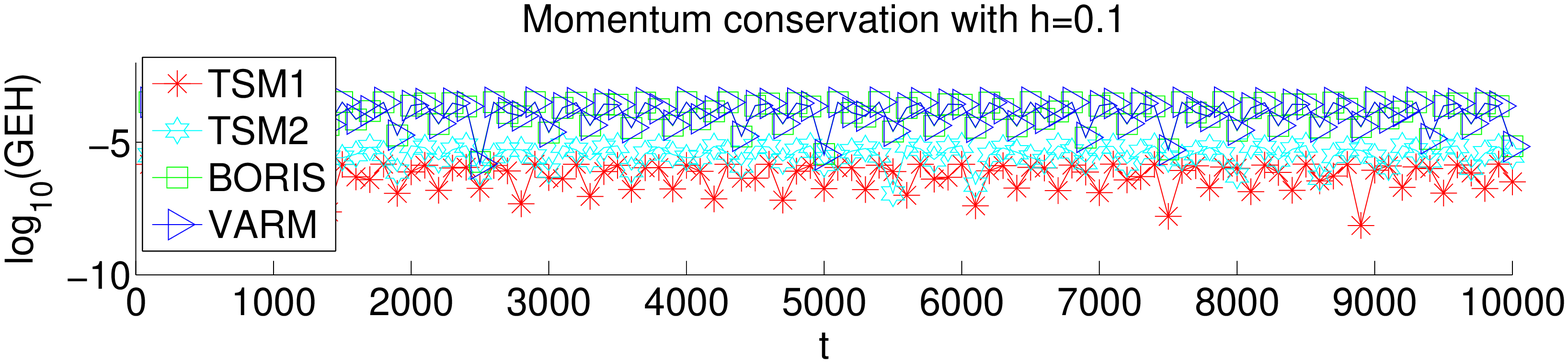}
\includegraphics[width=14cm,height=3cm]{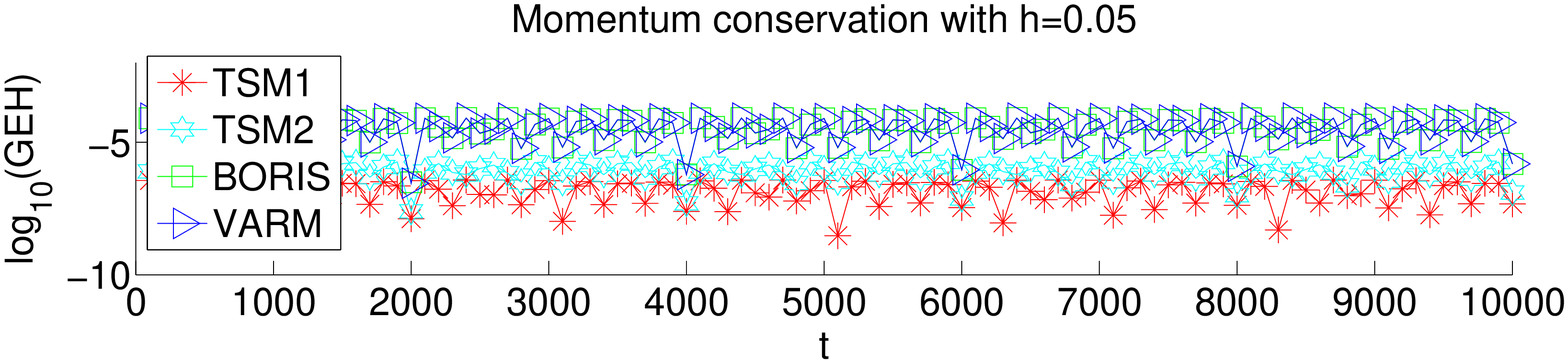}
\caption{The  errors of momentum against  $t$.} \label{p2}
\end{figure}

%\begin{figure}[ptb]
%\centering
%\includegraphics[width=6cm,height=7cm]{err1.eps}
%\includegraphics[width=6cm,height=7cm]{err2.eps}
%\caption{The  errors of energy against  $t$.} \label{p3}
%\end{figure}

% \textbf{Problem 2.}
%We consider another   potential $U(x)$ (see \cite{Hairer2017-1})
%$$U(x)=x_1^3-x_2^3+\frac{1}{5}x_1^4+x_2^4+x_3^4.$$
%and the magnetic field $B$ is the same as Problem 1. The initial
%values are taken as
%$$x(0)=(0,1,0.1)^{\intercal},\quad v(0)=(0.09,0.55,0.30)^{\intercal}.$$
%We   solve this problem in $[0,500000]$ with $h=0.01,0.005$ and the
%conservation  of   the  energy  is displayed in Figure \ref{p4}.
%
% \begin{figure}[ptb]
%\centering
%\includegraphics[width=14cm,height=3cm]{2H0.eps}
%\includegraphics[width=14cm,height=3cm]{2H1.eps}
%\caption{The  errors of energy against  $t$.} \label{p4}
%\end{figure}

%\begin{figure}[ptb]
%\centering
%\includegraphics[width=6cm,height=7cm]{2err1.eps}
%\includegraphics[width=6cm,height=7cm]{2err2.eps}
%\caption{The  errors of energy against  $t$.} \label{p5}
%\end{figure}

\subsection{Results in a strong field}
When the magnetic field is  strong, long time conservations can only
be obtained for the second method TSM2. In this subsection, we
present the results of the method TSM2  for charged-particle
dynamics in a strong magnetic field.
 We note
that if the magnetic field is further assumed to be a constant
field, TSM2 and TSM1 have the same scheme and thus they share  the
same long time behaviour.
\subsubsection{Near-conservation of energy.}
We define $$\xi(x)=2\arctan\big(\frac{h}{2\epsilon}\abs{B(x)}\big)$$
and consider the modified energy
$$H_h(x,v)=H(x,v)+(\xi\csc(\xi)-1)I(x,v)\abs{B(x)}.$$
%According to
%$$\xi\csc(\xi)=\frac{1}{\tan(\xi)}=\frac{1-\tan^2(\xi/2)}{2\tan(\xi/2)}
%=\frac{1-\big(\frac{h}{2\epsilon}\abs{B(x)}\big)^2}{2\big(\frac{h}{2\epsilon}\abs{B(x)}\big)}
%=\frac{\epsilon-
%\frac{h^2}{4\epsilon}\abs{B(x)}^2-h\abs{B(x)}}{h\abs{B(x)}},$$
%$H_h(x,v)$ can be formulated as
%$$H_h(x,v)=H(x,v)+\big(\frac{\epsilon}{h}-
%\frac{h}{4\epsilon}\abs{B(x)}^2-\abs{B(x)}\big) I(x,v).$$
\begin{theo}\label{nearly energy pre thm0s}
(\textbf{Energy conservation of TSM2.}) Suppose that  for $0 <
\epsilon \leq\epsilon_0$ and $0 < h \leq h_0$ with $ \epsilon_0> 0$
and $h_0
> 0$, the following holds. Let $N \geq 1$ be an arbitrary integer and the stepsize $h$ is
chosen such that $h\epsilon|B(x_n )| \leq2
\tan\big(\frac{\pi}{2(N+3)}\big)$ for some $N \geq 1$. It is assumed
further that the numerical solution  of the TSM2 stays in a compact
set $K$. Then TSM2 has the following near conservation of  modified
energy
$$\abs{H_h(x_{n+1/2},v_{n+1/2})-H_h(x_{1/2},v_{1/2})}\leq C\epsilon\ \ \ \textmd{for}\ \  nh\leq \epsilon^{-N},$$
where   the constant    $C$ is independent of $\epsilon, n$ and $h$
as long as $nh\leq h^{-N},$  but depends on $N$,  the bounds of the
$N + 1$ derivatives of $B$ and $E$ on the compact set $K$, and   the
constants in \eqref{ivb}.
\end{theo}

\subsubsection{Near-conservation of the magnetic moment}
We define the modified magnetic moment
$$I_h(x,v)=  \sec^2(\xi(x)/2)I(x,v)=  (\tan^2(\xi(x)/2)+1)I(x,v)
=\Big(\frac{h^2}{4\epsilon^2}\abs{B(x)}^2+1\Big)I(x,v).$$

\begin{theo}\label{nearly energy pre thm02s}
(\textbf{Magnetic moment conservation of TSM2.}) Under the
conditions of Theorem \ref{nearly energy pre thm02s},  the modified
magnetic moment is nearly conserved over long times
$$\abs{I_h(x_{n+1/2},v_{n+1/2})-I_h(x_{1/2},v_{1/2})}\leq C \epsilon\ \ \ \textmd{for}\ \  nh\leq \epsilon^{-N}.$$
\end{theo}

\begin{rem}
It is noted that the long term analysis of   the method VARM was
researched recently in   \cite{Hairer2018} for charged-particle
dynamics in a
 strong magnetic field. The conservations of a modified energy and a
 modified magnetic moment were proved there for VARM. It can be seen from Theorems \ref{nearly energy pre thm0s}-\ref{nearly energy pre thm02s}
 that TSM2  also has   modified energy and
 modified magnetic moment conservations. Moreover, according to   the analysis of
 \cite{Hairer2018} and the results of Theorems \ref{nearly energy pre thm0s}-\ref{nearly energy pre thm02s},
 for small $\xi$, the modified energy and
 modified magnetic moment of VARM and TSM2  are given respectively as
 \renewcommand\arraystretch{1.1}
$$\begin{tabular}{|c|c|c|}
  \hline
  Method & Modified Energy &  Modified Magnetic Moment\\
\hline
  VARM & $H+\frac{5}{12}\xi^2I\abs{B}$ & $(1+ \frac{1}{2}\xi^2)I$ \\
  TSM2 & $H+\frac{1}{6}\xi^2I\abs{B}$ & $(1+ \frac{1}{4}\xi^2)I$ \\
  \hline
\end{tabular}$$
This shows that compared with VARM,  the modified energy and
 modified magnetic moment that TSM2 nearly preserves  are closer to
 the original energy $H$ and magnetic moment $I$.
\end{rem}
\subsubsection{Experiment}
We still consider the problem given in Subsection \ref{subse:expe}
but with $\epsilon=0.01$. We  solve it in $[0,10000]$  by the second
method with $h=0.01, 0.005$. The conservations  of the modified
energy  and the modified magnetic moment are displayed in Figure
\ref{p3}.

 \begin{figure}[ptb]
\centering
\includegraphics[width=14cm,height=3cm]{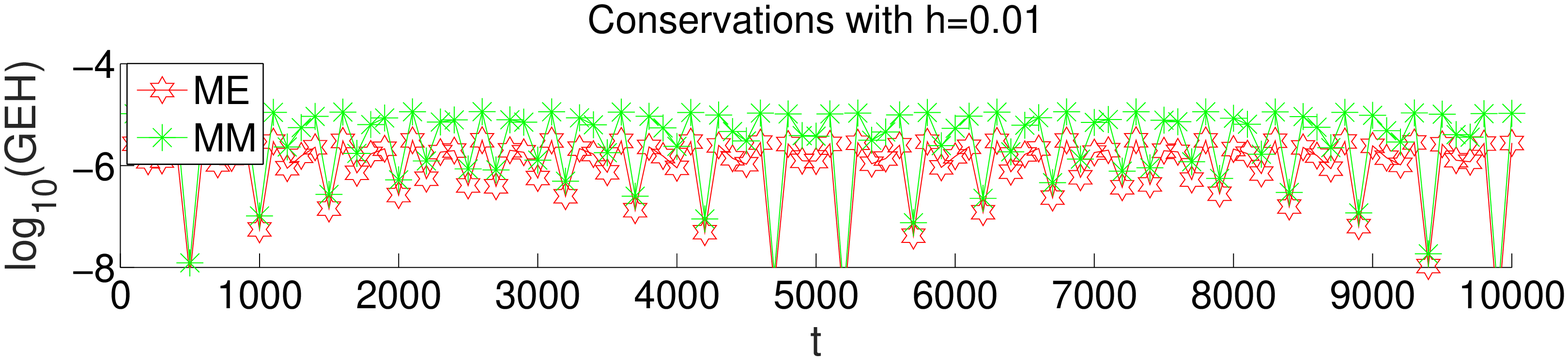}
\includegraphics[width=14cm,height=3cm]{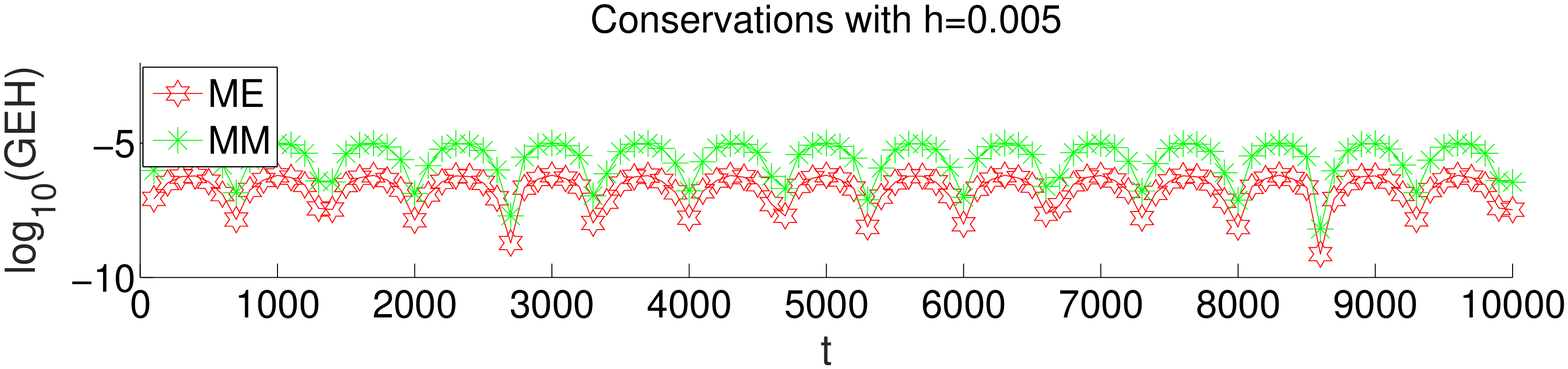}
\caption{The  errors of  modified energy (ME) and   modified
magnetic moment (MM) against  $t$.} \label{p3}
\end{figure}

\section{Proof of Theorem \ref{energy pre thm}}\label{sec:speci}
Under the special choice of $U$, it can be verified that
$$F\big(x_{n+1/2}\big)=\int_{0}^1 F\big(x_{n}+\sigma(x_{n+1}-x_{n}) \big) d\sigma.$$
Therefore, the method \eqref{TSM1} is identical to
\begin{equation}\label{M1 NEW SCHEME}
\left\{\begin{array}[c]{ll}x_{n+1}=x_{n}+
hv_{n+1/2},\\
v_{n+1}=v_{n}+h v_{n+1/2} \times \frac{1}{\epsilon} B(x_{n+1/2})+h
\int_{0}^1 F\big(x_{n}+\sigma(x_{n+1}-x_{n}) \big) d\sigma.
\end{array}\right.
\end{equation}
The system \eqref{charged-particle sts} as well as $v=\dot{x}$ can
be rewritten as
\begin{equation*}%\label{charged-sts-first order}
\begin{array}[c]{ll}
\frac{d}{dt }\left(
  \begin{array}{c}
    x \\
    v \\
  \end{array}
\right)  =\left(
            \begin{array}{cc}
              0 & I \\
              -I & \tilde{B}(x) \\
            \end{array}
          \right) \left(
  \begin{array}{c}
    -F(x) \\
    v \\
  \end{array}
\right)=\left(
            \begin{array}{cc}
              0 & I \\
              -I & \tilde{B}(x) \\
            \end{array}
          \right)\nabla E(x,v),
\end{array}
\end{equation*}
where the skew symmetric matrix $\tilde{B}$ is given by
$$\tilde{B}(x)=\frac{1}{\epsilon}\left(
                   \begin{array}{ccc}
                     0 & B_3(x) & -B_2(x) \\
                     -B_3(x) & 0 & B_1(x) \\
                     B_2(x) & -B_1(x) & 0 \\
                   \end{array}
                 \right).
$$
It is noted that this is a Poisson system and it has been shown in
\cite{Cohen2011} that the following scheme
\begin{equation}\label{EP M}
 \begin{aligned}\left(
  \begin{array}{c}
    x_{n+1} \\
    v_{n+1} \\
  \end{array}
\right) =\left(
  \begin{array}{c}
    x_{n} \\
    v_{n} \\
  \end{array}
\right)+h \left(
            \begin{array}{cc}
              0 & I \\
              -I & \tilde{B}\Big(\frac{x_{n}+x_{n+1}}{2}\Big) \\
            \end{array}
          \right) \left(
  \begin{array}{c}
    -\int_{0}^1F(x_n+\sigma
(x_{n+1}-x_n)) d \sigma\\
    \int_{0}^1(v_n+\sigma
(v_{n+1}-v_n)) d \sigma\\
  \end{array}
\right)
\end{aligned}
\end{equation}
preserves  the energy $E$ in \eqref{energy of cha} exactly. It is
easy to see that \eqref{M1 NEW SCHEME} and \eqref{EP M} have the
same scheme. Therefore, the proof is complete.

\section{Long term analysis in a normal magnetic field}\label{sec:normal}
In the analysis of this section, we will use   backward error
analysis (see Chap. IX of \cite{Hairer2006}).
\subsection{Proof of Theorem \ref{nearly energy pre thm0}}
Following the idea of backward error analysis, we search  for two
modified differential equations (as a formal series in powers of
$h$) such that their solutions $y(t)$ and $w(t)$ formally satisfy
$y(nh) = x_n$ and $w(nh) = v_n$, where $x_n$  and $v_n$ represent
the numerical solution obtained by the   first method \eqref{TSM1}.
According to the scheme of the first method, such functions have to
satisfy
\begin{equation}\label{TSM1-equation}
\left\{\begin{array}[c]{ll}y(t+h)=y(t)+
h\frac{w(t+h)+w(t)}{2},\\
w(t+h)=w(t)+h \frac{w(t+h)+w(t)}{2} \times
B\Big(\frac{y(t+h)+y(t)}{2}\Big)+h
 F\big(\frac{y(t+h)+y(t)}{2}\big).
\end{array}\right.
\end{equation}
From this scheme, we define
\begin{equation*}%\label{Operator}
\begin{array}[c]{ll}L_1(\zeta)=\zeta-1,\ \ \ L_2(\zeta)=\frac{\zeta+1}{2}.
\end{array}
\end{equation*}
Then \eqref{TSM1-equation} becomes
\begin{equation*}%\label{TSM1-equation1}
\left\{\begin{array}[c]{ll}L_1(e^{hD})y(t)=hL_2(e^{hD})w(t),\\
L_1(e^{hD})w(t)=h\big(L_2(e^{hD})w(t)\big) \times
B\big(L_2(e^{hD})y(t)\big)+h
 F\big(L_2(e^{hD})y(t)\big),
\end{array}\right.
\end{equation*}
where $D$  is the differential operator.  Letting
$$z(t)=L_2(e^{hD})y(t),$$ one obtains
\begin{equation*}%\label{TSM1-equation2}
\frac{1}{h^2}L^2_1L_2^{-2}(e^{hD})z(t)=\frac{1}{h}\big(L_1L_2^{-1}(e^{hD})z(t))\times
B(z)+F(z).
\end{equation*}
Using the following properties
\begin{equation}\label{Operator pro}
\begin{array}[c]{ll}
&L^2_1L_2^{-2}(e^{hD})=h^2D^2-\frac{1}{6}h^4D^4+\frac{17}{720}h^6D^6+\cdots,\\
&L_1L_2^{-1}(e^{hD})=hD-\frac{1}{12}h^3D^3+\frac{1}{120}h^5D^5+\cdots,\\
\end{array}
\end{equation}
we  obtain
\begin{equation}\label{TSM1-equation3}
\ddot{z}-\frac{1}{6}h^2z^{(4)}+\frac{17}{720}h^4z^{(6)}-\ldots=\Big(\dot{z}-\frac{1}{12}h^2z^{(3)}+\frac{1}{120}h^4z^{(5)}-\ldots
\Big)\times B(z)+F(z).
\end{equation}
This   gives the modified differential equation for $z$, which can
be rewritten as the following second-order differential equation
\begin{equation*}
\ddot{z}=\dot{z}\times
B(z)+F(z)+h^2G_2(z,\dot{z})+h^4G_4(z,\dot{z})+\ldots
\end{equation*}
with unique $h$-independent functions $G_{2j}(z,\dot{z})$.

 It is
remarked  that for even values of $k$ the expression
$\dot{z}^{\intercal}z^{ (k)}$ is a total differential. In the light
of   the fact that $a^{\intercal}\big(a\times B(b)\big)=0$ for any
$a,b$, a multiplication of \eqref{TSM1-equation3} with
$\dot{z}^{\intercal}$ implies
\begin{equation*}
 \begin{array}[c]{ll} \frac{\textmd{d}}{\textmd{d}t}
\Big(\frac{1}{2}\dot{z}^{\intercal}\dot{z}+U(z)-\frac{h^2}{6}\big(\dot{z}^{\intercal}z^{(3)}-\frac{1}{2}\ddot{z}^{\intercal}\ddot{z}\big)+\ldots\Big)
\\=h^2\dot{z}^{\intercal}\Big(\big(-\frac{1}{12}z^{(3)}+\frac{h^2}{120}z^{(5)}-\ldots\big)
\times
 B(z)\Big)+\mathcal{O}(h^N).
\end{array}
\end{equation*}
This shows that there exist $h$-independent functions $E_{2j} (x,v)$
such that the function
$$E_h(x,v)=E(x,v)+h^2E_{2} (x,v)+h^4E_{4} (x,v)+\cdots,$$
truncated at the $\mathcal{O}(h^N)$ term, satisfies
\begin{equation*}%\label{MOEN-1}
 \begin{array}[c]{ll} \frac{\textmd{d}}{\textmd{d}t}
 E_h(z,\dot{z})=h^2\dot{z}^{\intercal}\Big(\big(-\frac{1}{12}z^{(3)}+\frac{h^2}{120}z^{(5)}-\ldots\big) \times
 B(z)\Big)+\mathcal{O}(h^N)
\end{array}
\end{equation*}
along solutions of the modified differential equation
\eqref{TSM1-equation3}.

 If the magnetic field $B(x) = B$ is constant, it is clear that for odd values of $k$ the expression
$\dot{z}^{\intercal}\big(z^{(k)}\times B(z)\big)$ is a total
differential. Then  there exist $h$-independent functions
$\tilde{E}_{2j} (x,v)$ such that the function
$$\tilde{E}_h(x,v)=E(x,v)+h^2\tilde{E}_{2} (x,v)+h^4\tilde{E}_{4} (x,v)+\cdots,$$
truncated at the $\mathcal{O}(h^N)$ term, satisfies $
  \frac{\textmd{d}}{\textmd{d}t}
 \tilde{E}_h(z,\dot{z})=\mathcal{O}(h^N)
$ along solutions of  \eqref{TSM1-equation3}. By a standard argument
given in Chap. IX of \cite{Hairer2006}, Theorem \ref{nearly energy
pre thm0} is proved.

\subsection{Proof of Theorem \ref{nearly mom pre thm0}}
First of all, we note that $z^{\intercal}S(\dot{z}\times
B(z))=-\frac{\textmd{d}}{\textmd{d}t}(z^{\intercal}SA(z))$, $S
\nabla U(z)=0$ and for even values of $k$ the expression
$z^{\intercal}Sz^{(k)}$ is a total differential (see
\cite{Hairer2017-1}). Then by multiplying \eqref{TSM1-equation3}
with $z^{\intercal} S$, one obtains  that there exist
$h$-independent functions $M_{2j} (x,v)$ such that the function
$$M_h(x,v)=M(x,v)+h^2M_{2} (x,v)+h^4M_{4} (x,v)+\cdots,$$
truncated at the $\mathcal{O}(h^N)$ term, satisfies
\begin{equation*}%\label{MOEN-1}
 \begin{array}[c]{ll} \frac{\textmd{d}}{\textmd{d}t}
 M_h(z,\dot{z})=h^2z^{\intercal}S\Big(\big(-\frac{1}{12}z^{(3)}+\frac{h^2}{120}z^{(5)}-\ldots\big) \times
 B(z)\Big)+\mathcal{O}(h^N)
\end{array}
\end{equation*}
along solutions of  \eqref{TSM1-equation3}.
 For constant $B$, one gets $Sz=z\times B$
and for odd values of $k$, $(z\times B)^{\intercal}(z^{(k)}\times
B)$  is a total differential. Then there exist $h$-independent
functions $\tilde{M}_{2j} (x,v)$ such that
$$\tilde{M}_h(x,v)=M(x,v)+h^2\tilde{M}_{2} (x,v)+h^4\tilde{M}_{4} (x,v)+\cdots,$$
truncated at the $\mathcal{O}(h^N)$ term, satisfies $
 \frac{\textmd{d}}{\textmd{d}t}
 \tilde{M}_h(z,\dot{z})=\mathcal{O}(h^N)
$ along solutions of  \eqref{TSM1-equation3}. This completes the
proof of Theorem \ref{nearly mom pre thm0}.

\subsection{Proof of Theorem \ref{nearly energy pre thm02}}
For the second method, we still use $y(t)$ and $w(t)$  to denote the
solutions of   modified differential equations such that $y(nh) =
x_n$ and $w(nh) = v_n$. These equations read
\begin{equation*}%\label{TSM1-equation2}
\begin{array}[c]{ll}&y(t+h)-2y(t)+y(t-h)\\
=&\frac{h}{2}A'^{\intercal}(
(y(t+h)+y(t))/2)\big(y(t+h)-y(t)\big)\\& +
\frac{h}{2}A'^{\intercal}((y(t)+y(t-h))/2)
\big(y(t)-y(t-h)\big)  \\
&-h\big(A((y(t+h)+y(t))/2)-A((y(t)+y(t-h))/2)\big)\\
&+\frac{h^2}{2}\big(F((y(t+h)+y(t))/2)+F((y(t)+y(t-h))/2)\big).
\end{array}
\end{equation*}
By the operators $L_1,L_2$ and by letting $z(t)=L_2(e^{hD})y(t),$
these can be rewritten as
\begin{equation}\label{TSM1-equation32}
\frac{1}{h^2}L^2_1L_2^{-2}(e^{hD})z
=\frac{1}{h}A'^{\intercal}(z)(L_1L_2^{-1}(e^{hD})z)-
\frac{1}{h}L_1L_2^{-1}(e^{hD}) A(z)+F(z),
\end{equation}
which gives the modified differential equation
\begin{equation*}
\ddot{z}=\dot{z}\times
B(z)+F(z)+h^2G_2(z,\dot{z})+h^4G_4(z,\dot{z})+\ldots,
\end{equation*}
where $G_{2j}(z,\dot{z})$ for $j=1, \ldots$ are $h$-independent
functions determined uniquely.

 A
multiplication of \eqref{TSM1-equation32} with $\dot{z}^{\intercal}$
yields
\begin{equation}\label{TSM132}
 \begin{array}[c]{ll} \frac{1}{h^2}\dot{z}^{\intercal}
 (L^2_1L_2^{-2}(e^{hD})z)
=&\frac{1}{h}\dot{z}^{\intercal}\Big(A'^{\intercal}(z)(L_1L_2^{-1}(e^{hD})z)-L_1L_2^{-1}(e^{hD})
A(z)\Big)\\
&- \frac{\textmd{d}}{\textmd{d}t}U(z)+\mathcal{O}(h^N).
\end{array}
\end{equation}
The left-hand side is the time derivative of an expression in which
the appearing second and higher derivatives of $z$ can be
substituted as functions of $(z, \dot{z})$.  The first term in the
right-hand side of \eqref{TSM132} can also be written as the time
derivative of a function by using the way given in
\cite{Hairer2017-2} as follows. It is shown in \cite{Hairer2017-2}
that for a function $f$ that is analytic at $0$, partial integration
shows that for time-dependent smooth functions $u$ and $v$,
\begin{equation}\label{fact in Hairer}
\langle f(hD)u,v\rangle-\langle u,f(-hD)v\rangle\  \textmd{is\ a
total\ derivative\ up\ to}\ \mathcal{O}(h^N)\ \textmd{for\
arbitrary}\ N,
\end{equation}
where $\langle \cdot,\cdot\rangle$ is  Euclidean inner product.
Moreover, it is true that
$$\langle \dot{z},A'^{\intercal}(z)(L_1L_2^{-1}(e^{hD})z)\rangle=\langle A'(z)\dot{z},L_1L_2^{-1}(e^{hD})z\rangle
=\langle D A(z),L_1L_2^{-1}(e^{hD})z\rangle.$$ Thence the first term
in the right-hand side of \eqref{TSM132} becomes
$$\frac{1}{h}\big(\langle D A(z),L_1L_2^{-1}(e^{hD})z\rangle
-\langle L_1L_2^{-1}(e^{hD}) A(z),Dz\rangle\big),$$ which is  a
total  derivative  up  to $\mathcal{O}(h^{N-1})$ for  arbitrary $N$
by considering \eqref{fact in Hairer} with
$f(hD)=L_1L_2^{-1}(e^{hD})/(hD)$, $u=DA(z)$ and $v=Dz$
%%%%%%%%%%%%%%%%%%%%%%%%%%%%%%%%%%%%%%%%%%%%%%%%%%%%%%%%%%%%%%%%%%%%%%%%
\footnote{{   Clearly, it follows from \eqref{Operator pro} that
$f(-hD)=f(hD)$, which is used here. }}.
%%%%%%%%%%%%%%%%%%%%%%%%%%%%%%%%%%%%%%%%%%%%%%%%%%%%%%%%%%%%%%%%%%%%%%%%
Meanwhile, this function is of magnitude $\mathcal{O}(h^{2})$,
because
$\frac{1}{h}L_1L_2^{-1}(e^{hD})z=\dot{z}+\mathcal{O}(h^{2})$.

According to the above analysis,  there exist $h$-independent
functions $E_{2j} (x,v)$ such that the function
$$E_h(x,v)=E(x,v)+h^2E_{2} (x,v)+h^4E_{4} (x,v)+\cdots,$$
truncated at the $\mathcal{O}(h^N)$ term, satisfies
$\frac{\textmd{d}}{\textmd{d}t}
 E_h(z,\dot{z})=\mathcal{O}(h^N)
$ along solutions of the modified differential equation
\eqref{TSM1-equation32}. The proof is complete.

\subsection{Proof of Theorem \ref{nearly mom pre thm02}}
The proof is analogous to the previous proof when
\eqref{TSM1-equation32}  is multiplied  with $(Sz)^{\intercal}$. It
is noted that by using $A'(z)Sz=SA(z)$ (see \cite{Hairer2017-2}), we
have
%%%%%%%%%%%%%%%%%%%%%%%%%%%%%%%%%%%%%%%%%%%%%%%%%%%%%%%%%%%%%%%%%%%%%%%%
\footnote{{  We remark that
$L_1L_2^{-1}(e^{hD})=-L_1L_2^{-1}(e^{-hD})$ is used here. }}
%%%%%%%%%%%%%%%%%%%%%%%%%%%%%%%%%%%%%%%%%%%%%%%%%%%%%%%%%%%%%%%%%%%%%%%%
\begin{equation*}
 \begin{array}[c]{ll} &(Sz)^{\intercal}A'^{\intercal}(z)(L_1L_2^{-1}(e^{hD})z)-
(Sz)^{\intercal}L_1L_2^{-1}(e^{hD}) A(z)\\
=& (A'(z)Sz)^{\intercal} (L_1L_2^{-1}(e^{hD})z)+
 z ^{\intercal}L_1L_2^{-1}(e^{hD}) SA(z)\\
=& (SA(z))^{\intercal} (L_1L_2^{-1}(e^{hD})z)+
 z ^{\intercal}L_1L_2^{-1}(e^{hD}) SA(z)
 \\
=& \langle SA(z),L_1L_2^{-1}(e^{hD})z\rangle -\langle
L_1L_2^{-1}(e^{-hD}) SA(z),z\rangle,
\end{array}
\end{equation*}
which is  a total  derivative  up to $\mathcal{O}(h^N)$ for
arbitrary  $N$.
 According to $z^{\intercal}S(\dot{z}\times
B(z))=-\frac{\textmd{d}}{\textmd{d}t}(z^{\intercal}SA(z))$ and
$z^{\intercal}S \nabla U(z)=0$, it is obtained    that there exist
$h$-independent functions $M_{2j} (x,v)$ such that the function
$$M_h(x,v)=M(x,v)+h^2M_{2} (x,v)+h^4M_{4} (x,v)+\cdots,$$
truncated at the $\mathcal{O}(h^N)$ term, satisfies $
\frac{\textmd{d}}{\textmd{d}t}
 M_h(z,\dot{z})=\mathcal{O}(h^N)
$ along solutions of \eqref{TSM1-equation32}. This shows the result.

\section{Long term analysis in a strong magnetic field}\label{sec:strong}
In this section, we will use the technique of  modulated Fourier
expansion  for studying long-time conservation properties. Modulated
Fourier expansion was firstly developed in  \cite{Hairer00} and was
extended to systems with solution dependent frequency in
\cite{Hairer16,Hairer2018}. Using the novel modulated Fourier
expansion with state-dependent frequencies and eigenvectors recently
developed in \cite{Hairer2018}, we will derive the expansion of the
second method and show two almost-invariants of the modulated
functions.

\subsection{Modulated Fourier expansion}
\begin{prop}\label{MFE thm}
  For $0\leq nh\leq T_{\epsilon}$ with
$T_{\epsilon}=\mathcal{O}(\epsilon)$, it is assumed that the
numerical solution $x_n$ of the second method \eqref{TSM2} stays in
a compact set $K$ and
\begin{equation*}%\label{im con}
 \frac{h}{\epsilon} \abs{B(x_{n+1/2})}\leq
 2\tan\big(\frac{C}{2N+2}\big)\ \ \textmd{with}\  C< \pi
\end{equation*}
for some integer $N\geq 1$.  Then the numerical solution $x_{n+1/2}$
 admits the following   modulated Fourier expansion
\begin{equation*}
\begin{aligned} &x_{n+1/2}=  \sum\limits_{\abs{k} \leq N+1} \mathrm{e}^{\mathrm{i}k \phi
 (t)/\epsilon}\zeta^k(t)+\mathcal{O}(t^2\epsilon^2),\ \
 t=(n+\frac{1}{2})h,
\end{aligned}
%\label{MFE-ERKN}%
\end{equation*}
where the phase function  $\phi
 (t)$ is given by
\begin{equation}
\tan\big(\frac{1}{2}\eta
\dot{\phi}\big)=\frac{\eta}{2}\abs{B(\zeta^0)}.
\label{phi}%
\end{equation}
The bounds of the functions $\zeta^k(t)$ as well as their
derivatives (up to order $N$) are
\begin{equation*}
\zeta^k(t)=\mathcal{O}(\epsilon^{\abs{k}})\ \ \textmd{for \  all}\
\abs{k}\leq N+1
%\label{bound}%
\end{equation*}
and further one has
\begin{equation*}
\dot{\zeta}^0  \times B(\zeta^0)=\mathcal{O}(\epsilon),\ \
P_j(\zeta^0)\zeta^k=\mathcal{O}(\epsilon^2)\ \ \textmd{for}\
\abs{k}=1,\ j\neq k.
%\label{bound furth}%
\end{equation*}
The functions $\zeta^k(t)$ are unique up to
$\mathcal{O}(\epsilon^{N+2})$ and the constants symbolised by
$\mathcal{O}$  independent of $n$ and $\epsilon$ as long as $0 \leq
nh \leq T_{\epsilon}$, but depend on $N, T$, the constants in
\eqref{ivb}, and the bounds of derivatives of $A(x)$ and $U(x).$ The
  orthogonal projections $P_j$ onto the eigenspaces are referred to
  \cite{Hairer2018}.
\end{prop}

\begin{proof} The   proof is analogous to the proof in Section 5 of
\cite{Hairer2018} but with some necessary  modifications. For
brevity we only present the main differences and for the details we
refer to \cite{Hairer2018}.

 As previous analysis, the second method \eqref{TSM2} has the following scheme
  \begin{equation*}%\label{TSM2 strong}
\frac{1}{h^2}L^2_1L_2^{-2}(e^{hD})z_n =\frac{1}{\epsilon
h}A'^{\intercal}(z_n)(L_1L_2^{-1}(e^{hD})z_n)- \frac{1}{\epsilon
h}L_1L_2^{-1}(e^{hD}) A(z_n)+F(z_n),
\end{equation*}
where $z_n=\frac{x_n+x_{n+1}}{2}.$
%Moreover,   because of $v\times
%\frac{1}{\epsilon} B(x)=\big(A'^{\intercal}(x)-A'(x))v$, it can also
%be written as
%\begin{equation}\label{TSM2-ano strong}
%\frac{1}{h^2}L^2_1L_2^{-2}(e^{hD})z_n =\frac{1}{\epsilon
%h}L_1L_2^{-1}(e^{hD})z_n \times B(z_n)+F(z_n)+\frac{1}{\epsilon h}
%\big(A'^{\intercal}(z_n)L_1L_2^{-1}(e^{hD})z_n-L_1L_2^{-1}(e^{hD})A(z_n)\big).
%\end{equation}
For the  solution $z_n$ we consider the following  modulated Fourier
expansion
\begin{equation*}%\label{MFE}
 z_n\approx \sum\limits_{k\in \ZZ} \mathrm{e}^{\mathrm{i}k \phi
 (t)/\epsilon}\zeta^k(t)=
 \sum\limits_{k\in \ZZ} \Upsilon^k(t),
\end{equation*}
where $\Upsilon^k(t)=\mathrm{e}^{\mathrm{i}k \phi
 (t)/\epsilon}\zeta^k(t)$, $t=t_{n+1/2}:=\frac{t_n+t_{n+1}}{2}$, and the  functions $\phi$ and $\zeta^k$ depend on the stepsize $h$ and  $\eta=h/\epsilon.$

For the operators $\frac{1}{h}L_1L_2^{-1}(e^{hD})$ and
$\frac{1}{h^2}L^2_1L_2^{-2}(e^{hD})$, with careful computations  it
is obtained that
\begin{equation*}
 \begin{array}[c]{ll}
  &\frac{1}{h}L_1L_2^{-1}(e^{hD})\Upsilon^k(t)=\mathrm{e}^{\mathrm{i}k \phi
 (t)/\epsilon}\sum\limits_{l\geq 0}\epsilon^{l-1}c_l^k
 \frac{\textmd{d}^l}{\textmd{d}t^l}\zeta^k(t),\\
 &\frac{1}{h^2}L^2_1L_2^{-2}(e^{hD})\Upsilon^k(t)=\mathrm{e}^{\mathrm{i}k \phi
 (t)/\epsilon}\sum\limits_{l\geq 0}\epsilon^{l-2}d_l^k
 \frac{\textmd{d}^l}{\textmd{d}t^l}\zeta^k(t),\\
\end{array}
\end{equation*}
where $c^0_{2j}=d^0_{0}=d^0_{2j+1}=0$ and $c^0_{2j+1}=\alpha_{2j+1}
\eta^{2j},\ d^0_{2j}=\beta_{2j} \eta^{2j-2}$. Here $\alpha$ and
$\beta$ are the   coefficients appearing in the following series
$$2\tanh(t/2)=\sum\limits_{j\geq 0}^{\infty}\alpha_j t^j,\ \ 4\tanh^2(t/2)=\sum\limits_{j\geq 0}^{\infty}\beta_j t^j.$$
The first coefficients for $k \neq 0$ are defined by
\begin{equation}\label{cd coeffi}
 \begin{array}[c]{ll}
  &c_0^k=\frac{\mathrm{i}}{\eta}2\tan\big(\frac{1}{2}k\eta \dot{\phi}\big)+\frac{1}{2}\epsilon k \eta
  \tan\big(\frac{1}{2}k\eta \dot{\phi}\big) \sec^2\big(\frac{1}{2}k\eta \dot{\phi}\big)\ddot{\phi}+\mathcal{O}(\epsilon),\\
  &c_1^k=  \sec^2\big(\frac{1}{2}k\eta \dot{\phi}\big)+\mathcal{O}(\epsilon),\\
    &d_0^k=\frac{-1}{\eta^2}4\tan^2\big(\frac{1}{2}k\eta \dot{\phi}\big)+\frac{\mathrm{i}\epsilon k\ddot{\phi}}{2}
    (2-\cos(k\eta \dot{\phi})) \sec^4\big(\frac{1}{2}k\eta \dot{\phi}\big)+\mathcal{O}(\epsilon),\\
    &d_1^k=\frac{2\mathrm{i}}{\eta}2\tan\big(\frac{1}{2}k\eta \dot{\phi}\big) \sec^2\big(\frac{1}{2}k\eta
    \dot{\phi}\big)+\mathcal{O}(\epsilon).
\end{array}
\end{equation}
It is noted that these coefficients depend on $\epsilon, \eta$ and
$t$ via derivatives of $\phi(t)$.

 For the second method, from the following required condition given in
  \cite{Hairer2018} $$\pm
\mathrm{i}c_0^{\pm1}\abs{B(\zeta^0)}=d_0^{\pm1},$$ the result
\eqref{phi} is obtained. The $e^{\pm1}_{\pm1}$ and $e^{j}_{k}$
presented in \cite{Hairer2018} are replaced by
\begin{equation*}
 \begin{array}[c]{ll}
&e^{\pm1}_{\pm1}=\pm\frac{2
\mathrm{i}}{\eta}\tan\big(\frac{1}{2}\eta \dot{\phi})
\sec^2\big(\frac{1}{2}\eta \dot{\phi}\big),\\
&e^{j}_{k}=-\frac{4 }{\eta^2}\tan\big(\frac{1}{2}k\eta
\dot{\phi})\big(\tan\big(\frac{1}{2}k\eta \dot{\phi})- j
\tan\big(\frac{1}{2}\eta \dot{\phi}\big)\big).
\end{array}
\end{equation*}

The rest parts of the proof are the same as those presented in
Section 5 of \cite{Hairer2018}. By the same arguments stated there,
the proof is complete.
\end{proof}
\subsection{Proof of Theorem \ref{nearly energy pre thm0s}}
\begin{prop}\label{H invariant thm}
Under the conditions of Proposition \ref{MFE thm}, there exists a
function $\mathcal{H}[\zeta]$ such that
\begin{equation*}
\begin{aligned}
&\mathcal{H}[\zeta](t)=\mathcal{H}[\zeta](0)+\mathcal{O}(t\epsilon^{N}) \ \ \textmd{for}\ 0\leq t\leq T_{\epsilon},\\
&\mathcal{H}[\zeta](t_{n+1/2})=H_h(x_{n+1/2},v_{n+1/2})+\mathcal{O}(\epsilon)
\ \ \textmd{for}\ nh\leq  T_{\epsilon},
\label{HH}%
\end{aligned}
\end{equation*}
where the constants  symbolised by $\mathcal{O}$ are independent of
$n,\ h$ and $\epsilon$, but depends on $N$.
\end{prop}
\begin{proof}
 $\bullet$ \textbf{Proof of the first statement.}  For $\abs{k} > N +
1$, it is assumed that $\Upsilon^k(t) = 0$. The equation \eqref{TSM2
strong} for the modulation functions can be written as
\begin{equation}\label{TSM2 strong}\begin{aligned}
\frac{1}{h^2}L^2_1L_2^{-2}(e^{hD})\Upsilon^k =&\frac{1}{\epsilon
}\Big( \sum\limits_{j\in \ZZ}\big(\frac{\partial
\mathcal{A}_j}{\partial \Upsilon^k} (\Upsilon)\big)^{*}
\frac{L_1L_2^{-1}(e^{hD})}{h}\Upsilon^j-\frac{L_1L_2^{-1}(e^{hD})}{h}
\mathcal{A}_k(\Upsilon) \Big)\\
&-\big(\frac{\partial \mathcal{U}}{\partial \Upsilon^k}
(\Upsilon)\big)^{*}+\mathcal{O}(\epsilon^N),
\end{aligned}\end{equation}
where $\mathcal{A}_j(\Upsilon)$ and $\mathcal{U}(\Upsilon)$ are
given by (see \cite{Hairer2018})
\begin{equation*}%\label{extended potential}
\begin{aligned}
&\mathcal{U}(\Upsilon)=\sum\limits_{0\leq m \leq N+1, s(\alpha)=0}
\frac{U^{(m)}}{m!}(\Upsilon^0) \Upsilon^{\alpha},\\
&\mathcal{A}(\Upsilon)=\Big(\sum\limits_{0\leq m \leq N+2,
s(\alpha)=k} \frac{A^{(m)}}{m!}(\Upsilon^0)
\Upsilon^{\alpha}\Big)_{k\in
\ZZ}=\Big(\mathcal{A}_k(\Upsilon)\Big)_{k\in \ZZ}.
\end{aligned}
\end{equation*}
Multiplication of \eqref{TSM2 strong} with
$(\dot{\Upsilon}^{k})^{*}$ and summation over $k$ yields
\begin{equation}\label{TSM2 strong-0}
\begin{aligned}
 \sum\limits_{k}& (\dot{\Upsilon}^{k})^{*}
\frac{1}{h^2}L^2_1L_2^{-2}(e^{hD})\Upsilon^k +\frac{1}{\epsilon
}\sum\limits_{k}\Big( \frac{\textmd{d} }{\textmd{d} t }
\mathcal{A}_k(\Upsilon)^{*} \frac{L_1L_2^{-1}(e^{hD})}{h}\Upsilon^k\\
&- (\dot{\Upsilon}^{k})^{*}\frac{L_1L_2^{-1}(e^{hD})}{h}
\mathcal{A}_k(\Upsilon) \Big)+ \frac{\textmd{d} }{\textmd{d} t
}\mathcal{U} (\Upsilon)=\mathcal{O}(\epsilon^N).
\end{aligned}
\end{equation}

By the expansion \eqref{Operator pro} of the operator
$\frac{1}{h^2}L^2_1L_2^{-2}$ and the analysis given in Theorem 5.1
of \cite{Hairer16}, we know that the first sum is a total
derivative. The second sum is a total differential by the proof  of
Theorem \ref{nearly energy pre thm02} given in this paper.
Therefore, the left-hand side   is a total derivative of function
$\mathcal{H}[\zeta]$ such that
$$\frac{\textmd{d} }{\textmd{d} t } \mathcal{H}[\zeta]=\mathcal{O}(\epsilon).$$
The first result of the theorem is shown.

 \vskip1mm $\bullet$ \textbf{Proof of the second statement.} In what follows, we prove the second statement of the theorem. To
this end, one needs to determine the dominant part of
$$\mathcal{H}[\zeta](t)=\mathcal{K}[\zeta](t)+\mathcal{M}[\zeta](t)+\mathcal{U}(\zeta(t)),$$
 where the time derivatives of  $\mathcal{K},\mathcal{M},\mathcal{U}$ equal the three corresponding terms on the left-hand
side of \eqref{TSM2 strong-0}.

Firstly, it is clear that
 $$\mathcal{U}[\zeta]=U(\zeta^0)+\mathcal{O}(\epsilon).$$

 For $\mathcal{K}[\zeta]$, we compute
\begin{equation*}%\label{extended potential}
\begin{aligned}
& \sum\limits_{k} (\dot{\Upsilon}^{k})^{*}
\frac{1}{h^2}L^2_1L_2^{-2}(e^{hD})\Upsilon^k=\sum\limits_{k}
\big(\dot{\zeta}^{k}+\frac{\mathrm{i}k
\dot{\phi}}{\epsilon}\zeta^k\big)^{*}\big(\sum\limits_{l\geq
0}\epsilon^{l-2}d_l^k
 \frac{\textmd{d}^l}{\textmd{d}t^l}\zeta^k(t)\big)\\
=&(\dot{\zeta}^0)^{*} \ddot{\zeta}^0+\sum\limits_{k=\pm1}\Big(
\big(\frac{\mathrm{i}k \dot{\phi}}{\epsilon}\zeta^k
\big)^{*}\frac{1}{\epsilon^2}d_0^k \zeta^k+
(\dot{\zeta}^{k})^{*}\frac{1}{\epsilon^2}d_0^k
\zeta^k+\big(\frac{\mathrm{i}k \dot{\phi}}{\epsilon}\zeta^k
\big)^{*}\frac{1}{\epsilon}d_1^k
\dot{\zeta}^k\Big)+\mathcal{O}(\epsilon)\\
=&(\dot{\zeta}^0)^{*}
\ddot{\zeta}^0+\frac{2\dot{\phi}\ddot{\phi}}{\epsilon^2}
 (2-\cos(\eta \dot{\phi})) \sec^4\big(\frac{1}{2}\eta
 \dot{\phi}\big)\abs{\zeta^1}^2\\
 &+\frac{4}{\epsilon^2\eta^2}\tan\big(\frac{1}{2}\eta \dot{\phi}\big)
\big(\eta \dot{\phi}\sec^2\big(\frac{1}{2}\eta
 \dot{\phi}\big)- \tan\big(\frac{1}{2}\eta \dot{\phi}\big)\big)
\big((\dot{\zeta}^1)^{*}
\zeta^1+(\zeta^1)^{*}\dot{\zeta}^1\big)+\mathcal{O}(\epsilon).
\end{aligned}
\end{equation*}
It can be checked that
$$\frac{2}{\eta^2}\frac{\textmd{d} }{\textmd{d} t }\tan\big(\frac{1}{2}\eta \dot{\phi}\big)
\big(\eta \dot{\phi}\sec^2\big(\frac{1}{2}\eta
 \dot{\phi}\big)-\tan \big(\frac{1}{2}\eta \dot{\phi}\big)\big)=\dot{\phi}\ddot{\phi}
 (2-\cos(\eta \dot{\phi})) \sec^4\big(\frac{1}{2}\eta
 \dot{\phi}\big).$$
Therefore, we have
$$\mathcal{K}[\zeta]=\frac{1}{2}\abs{\zeta^0}^2+\frac{4\abs{\zeta^1}^2}{\epsilon^2\eta^2}\tan\big(\frac{1}{2}\eta \dot{\phi}\big)
\big(\eta \dot{\phi}\sec^2\big(\frac{1}{2}\eta
 \dot{\phi}\big)- \big(\frac{1}{2}\eta \dot{\phi}\big)\big)+\mathcal{O}(\epsilon).$$

We then turn to $\mathcal{M}[\zeta]$. Since the term $k = 0 $  is of
size $\mathcal{O}(\epsilon)$,  the dominating terms appear for $k =
\pm1$, which read
\begin{equation*}%\label{TSM2 strong-1}
\begin{aligned}
&\frac{1}{\epsilon }\sum\limits_{k=\pm 1}\Big( \frac{\textmd{d}
}{\textmd{d} t } \mathcal{A}_k(\Upsilon)^{*}
\frac{L_1L_2^{-1}(e^{hD})}{h}\Upsilon^k-
(\dot{\Upsilon}^{k})^{*}\frac{L_1L_2^{-1}(e^{hD})}{h}\partial
\mathcal{A}_k(\Upsilon) \Big)\\
=&\frac{\mathrm{i}\dot{\phi}}{\epsilon^3}(c_0^1+c_0^{-1})\big((A'(\zeta^0)\zeta^1)^*\zeta^1-(\zeta^1)^*A'(\zeta^0)\zeta^1\big)\\
+&\frac{1}{\epsilon^2}(c_0^1-\mathrm{i}\dot{\phi}c_0^{-1})\big(
\frac{\textmd{d} }{\textmd{d} t }(A'(\zeta^0)\zeta^1)^*\zeta^1+
(A'(\zeta^0)\zeta^1)^*\dot{\zeta}^1\\
&-(\dot{\zeta}^1)^*A'(\zeta^0)\zeta^1 -(\zeta^1)^*\frac{\textmd{d}
}{\textmd{d} t }(A'(\zeta^0)\zeta^1)\big)+\mathcal{O}(\epsilon).
\end{aligned}
\end{equation*}
It follows from \cite{Hairer2018} that
$$(A'(\zeta^0)\zeta^1)^*\zeta^1-(\zeta^1)^*A'(\zeta^0)\zeta^1=\mathrm{i}\abs{B(\zeta^0)}\abs{\zeta^1}^2+\mathcal{O}(\epsilon).$$
Moreover, according to \eqref{cd coeffi}, we have
\begin{equation*}
\begin{aligned}
&c_0^1+c_0^{-1}=\epsilon \eta \tan\big(\frac{1}{2}\eta
 \dot{\phi}\big) \sec^2\big(\frac{1}{2}\eta
 \dot{\phi}\big)\ddot{\phi}+\mathcal{O}(\epsilon^2),\\
 &c_0^1-\mathrm{i}\dot{\phi}c_0^{-1}=\mathrm{i}\dot{\phi}\Big(\frac{\tan\big(\frac{1}{2}\eta
 \dot{\phi}\big)}{\frac{1}{2}\eta
 \dot{\phi}}-\sec^2\big(\frac{1}{2}\eta
 \dot{\phi}\big)\Big)+\mathcal{O}(\epsilon)
\end{aligned}
\end{equation*}
and it can be checked that the following relation holds
$$\frac{\textmd{d} }{\textmd{d} t }\dot{\phi}\Big(\frac{\tan\big(\frac{1}{2}\eta
 \dot{\phi}\big)}{\frac{1}{2}\eta
 \dot{\phi}}-\sec^2\big(\frac{1}{2}\eta
 \dot{\phi}\big)\Big)= \frac{\dot{\phi}}{\epsilon}(c_0^1+c_0^{-1}).$$
Based on these analysis,   the dominating terms  of
$\mathcal{M}[\zeta]$ are
$$\mathcal{M}[\zeta]=\Big(\frac{\tan\big(\frac{1}{2}\eta
 \dot{\phi}\big)}{\frac{1}{2}\eta
 \dot{\phi}}-\sec^2\big(\frac{1}{2}\eta
 \dot{\phi}\big)\Big)\dot{\phi}\abs{B(\zeta^0)}\frac{\abs{\zeta^1}^2}{\epsilon^2}+\mathcal{O}(\epsilon).$$

Therefore, $\mathcal{H}[\zeta]$ is obtained as follows
\begin{equation*}
\begin{aligned}
\mathcal{H}[\zeta](t)=&\frac{1}{2}\abs{\zeta^0}^2+\frac{4\abs{\zeta^1}^2}{\epsilon^2\eta^2}\tan\big(\frac{1}{2}\eta
\dot{\phi}\big) \big(\eta \dot{\phi}\sec^2\big(\frac{1}{2}\eta
 \dot{\phi}\big)- \tan\big(\frac{1}{2}\eta \dot{\phi}\big)\big)\\
 &+\Big(\frac{\tan\big(\frac{1}{2}\eta
 \dot{\phi}\big)}{\frac{1}{2}\eta
 \dot{\phi}}-\sec^2\big(\frac{1}{2}\eta
 \dot{\phi}\big)\Big)\dot{\phi}\abs{B(\zeta^0)}\frac{\abs{\zeta^1}^2}{\epsilon^2}+U(\zeta^0)+\mathcal{O}(\epsilon).
\end{aligned}
\end{equation*}

On the other hand, we  consider the dominant term in
$$E(x_{n+\frac{1}{2}},v_{n+\frac{1}{2}})=\frac{1}{2}\abs{v_{n+\frac{1}{2}}}^2+U(x_{n+\frac{1}{2}}).$$
Inserting the modulated Fourier expansion into $x_{n+\frac{1}{2}}$
yields  $x_{n+\frac{1}{2}} =  \zeta^0(t_{n+\frac{1}{2}} )
  +\mathcal{O}(\epsilon)$ and further $U(x_{n+\frac{1}{2}})=U(\zeta^0(t_{n+\frac{1}{2}})) +\mathcal{O}(\epsilon)$.
It is assumed that the modulated Fourier expansion for
$v_{n+\frac{1}{2}}$ is
\begin{equation*}%\label{MFE}
 v_{n+\frac{1}{2}} \approx \sum\limits_{k\in \ZZ} \mathrm{e}^{\mathrm{i}k \phi
 (t)/\epsilon}\vartheta^k(t).
\end{equation*}
From
$$v_{n+\frac{1}{2}}=\frac{L_1L_2^{-1}(e^{hD})}{h}x_{n+\frac{1}{2}}=\sum\limits_{k\in \ZZ}\frac{L_1L_2^{-1}(e^{hD})}{h} \mathrm{e}^{\mathrm{i}k \phi
 (t)/\epsilon}\zeta^k(t),$$
it follows that
\begin{equation*}
\begin{aligned}
&\vartheta^0=\dot{\zeta}^0+\mathcal{O}(\epsilon),\\
&\vartheta^1=\frac{1}{\epsilon}c_0^1\zeta^1+\mathcal{O}(\epsilon)=\frac{\mathrm{i}}{\epsilon
\eta}
2\tan\big(\frac{1}{2}\eta \dot{\phi}\big)\zeta^1+\mathcal{O}(\epsilon),\\
&\vartheta^{-1}=\frac{1}{\epsilon}c_0^{-1}\zeta^{-1}+\mathcal{O}(\epsilon)=-\frac{\mathrm{i}}{\epsilon
\eta} 2\tan\big(\frac{1}{2}\eta
\dot{\phi}\big)\zeta^{-1}+\mathcal{O}(\epsilon).
\end{aligned}
\end{equation*}
Thus, we have
$$v_{n+\frac{1}{2}}=\dot{\zeta}^0+\frac{\mathrm{i}}{h}2\tan\big(\frac{1}{2}\eta
\dot{\phi}\big)\big(\zeta_1^1\mathrm{e}^{\mathrm{i} \phi
 (t)/\epsilon}-\zeta_{-1}^{-1}\mathrm{e}^{-\mathrm{i} \phi
 (t)/\epsilon}\big)+\mathcal{O}(\epsilon),$$
which yields that \begin{equation*}
\begin{aligned}
\frac{1}{2}\abs{v_{n+\frac{1}{2}}}^2=\frac{1}{2}\abs{\dot{\zeta}^0}^2+\frac{4\tan^2\big(\frac{1}{2}\eta
\dot{\phi}\big)}{h^2}\abs{\zeta_1^1}^2
\end{aligned}
\end{equation*}
and
\begin{equation*}
\begin{aligned}
I(x_{n+\frac{1}{2}},v_{n+\frac{1}{2}})=\frac{4\tan^2\big(\frac{1}{2}\eta
\dot{\phi}\big)}{h^2}\frac{\abs{\zeta_1^1}^2}{\abs{B(\zeta^0)}}.
\end{aligned}
\end{equation*}
Therefore, we obtain
\begin{equation*}
\begin{aligned}E(x_{n+\frac{1}{2}},v_{n+\frac{1}{2}})&=\frac{1}{2}\abs{\dot{\zeta}^0}^2+
\frac{4\tan^2\big(\frac{1}{2}\eta \dot{\phi}\big)}{h^2}
\abs{\zeta_1^1}^2 +U(\zeta^0(t_{n+\frac{1}{2}}))
+\mathcal{O}(\epsilon)\\
&=\frac{1}{2}\abs{\dot{\zeta}^0}^2+\abs{B(\zeta^0)}I(x_{n+\frac{1}{2}},v_{n+\frac{1}{2}})
+U(\zeta^0(t_{n+\frac{1}{2}})) +\mathcal{O}(\epsilon).
\end{aligned}
\end{equation*}
By subtracting the expressions obtained for
$\mathcal{H}[\zeta](t_{n+\frac{1}{2}})$ and
$E(x_{n+\frac{1}{2}},v_{n+\frac{1}{2}})$ from each other, one gets
\begin{equation*}
\begin{aligned}\mathcal{H}[\zeta](t_{n+\frac{1}{2}})-E(x_{n+\frac{1}{2}},v_{n+\frac{1}{2}})=
(\eta \dot{\phi}\csc(\eta
\dot{\phi})-1)I(x_{n+\frac{1}{2}},v_{n+\frac{1}{2}})\abs{B(\zeta^0)}+\mathcal{O}(\epsilon),
\end{aligned}
\end{equation*}
where \eqref{phi} is used in the simplifications.  The second
statement is obtained immediately from this result.
\end{proof}

According to the above results and following the analysis presented
in \cite{Hairer2018} and  Chap. XIII of \cite{Hairer2006}, the
statement of Theorem \ref{nearly energy pre thm0s} is easily
obtained.

%%%%%%%%%%%%%%%%%%%%%%%%%%%%%%%%%%%%%%%%%%%%

\subsection{Proof of Theorem \ref{nearly energy pre thm02s}}
\begin{prop}\label{second invariant thm}
Under the conditions of Proposition \ref{MFE thm}, there exists a
function $\mathcal{I}[\zeta]$ such that
\begin{equation*}
\begin{aligned}
&\mathcal{I}[\zeta](t)=\mathcal{I}[\zeta](0)+\mathcal{O}(t\epsilon^{N}) \ \ \textmd{for}\ 0\leq t\leq T_{\epsilon},\\
&\mathcal{I}[\zeta](t_{n+1/2})=I_h(x_{n+1/2},v_{n+1/2})+\mathcal{O}(\epsilon)
\ \ \textmd{for}\ nh\leq  T_{\epsilon},
\label{HH}%
\end{aligned}
\end{equation*}
where the constants  symbolised by $\mathcal{O}$ are independent of
$n,\ h$ and $\epsilon$, but depends on $N$.
\end{prop}
\begin{proof}
  $\bullet$ \textbf{Proof of the first statement.} Multiplication of \eqref{TSM2 strong} with $-\ii
k(\Upsilon^{k})^{*}$ and summation over $k$ yields
\begin{equation}\label{TSM2 strong-02}
\begin{aligned}
 -\sum\limits_{k} \ii k(\Upsilon^{k})^{*}
\frac{1}{h^2}L^2_1L_2^{-2}(e^{hD})\Upsilon^k +\frac{1}{\epsilon
}\sum\limits_{k}\ii k\Big( &\mathcal{A}_k(\Upsilon)^{*}
\frac{L_1L_2^{-1}(e^{hD})}{h}\Upsilon^k\\
&- (\Upsilon^{k})^{*}\frac{L_1L_2^{-1}(e^{hD})}{h}
\mathcal{A}_k(\Upsilon) \Big)=\mathcal{O}(\epsilon),
\end{aligned}
\end{equation}
where the results (4.43) and (4.44) of \cite{Hairer2018} are used
here. Similarly to the analysis of previous section, it can be shown
that the real part of this left-hand size  is a total derivative.
Therefore, there exists a function $\mathcal{I}[\zeta]$ such that
$\frac{\textmd{d} }{\textmd{d} t
}\mathcal{I}[\zeta]=\mathcal{O}(\epsilon).$ This proves the first
statement of the theorem.

 \vskip1mm $\bullet$ \textbf{Proof of the second statement.}  Concerning the dominant terms of $\mathcal{I}[\zeta]$, the left-hand
size of \eqref{TSM2 strong-02} is zero for $k=0$. For $k=\pm1$, we
can verify that the second sum is of size $\mathcal{O}(\epsilon)$.
Thus, the dominant part of $\mathcal{I}[\zeta]$ only exist in the
first sum for $k=\pm1$. With this and  the ``magic formulas" on p.
508 of \cite{Hairer2006}, we deduce that
\begin{equation}\label{ld1}
\begin{aligned}
& \mathrm{Re}(\textmd{the\ first\ sum\  of\  \eqref{TSM2 strong-02}}) \\
%=&\frac{\mathrm{i}}{2\epsilon} \sum\limits_{|k|\leq
%N+1}k\mathrm{Im}\Big[\big(\bar{q}_{h} \big)^\intercal
%\cos(\frac{1}{2}h\Omega)(-h^2\bar{b}_1(h\Omega))^{-1}
% 4\sin^2(\frac{1}{2}h\Omega)q_{h}  \\
% &+ \big(\bar{q}_{h} \big)^\intercal
%\cos(\frac{1}{2}h\Omega)(-h^2\bar{b}_1(h\Omega))^{-1}
%  \sum\limits_{l\geq0}\frac{2h^{2l+2}}{(2l+2)!}
% q^{(2l+2)}_{h} \Big)\Big]\\
=&-\ii
\sum\limits_{k}k\frac{d}{dt}\sum\limits_{l\geq0}\beta_{2l}h^{2l}\mathrm{Im}\Big[\big(
\Upsilon^k \big)^*(\Upsilon^k)^{(2l+1)}
                                               -\big(\dot{\Upsilon^k}
                                               \big)^*
                                                (\Upsilon^k)^{(2l)}
+\cdots\pm  \big((\Upsilon^k)^{(l)} \big)^*
                                               (\Upsilon^k)^{(l+1)}  \Big].
\end{aligned}
\end{equation}
 For $k \neq 0$, from Lemma 5.1 given in \cite{Hairer16}, it follows
 that
\begin{equation*}\begin{aligned}
\frac{1}{m!}\frac{d^m}{dt^m}\Upsilon^k(t)=\frac{1}{m!}\zeta^k(t)
\big(\frac{\mathrm{i}k}{\epsilon}\dot{\phi}(t)\big)^m
\mathrm{e}^{\mathrm{i}k
\phi(t)/\epsilon}+\mathcal{O}\Big(\frac{1}{(m/M)!}\big(\frac{c}{\epsilon}\big)^{m-1-|k|}
\Big),
%\label{boun-der}%
\end{aligned}
\end{equation*}
where $c$ and the constant symbolised by $\mathcal{O}$ are
independent of $m  \geq 1$ and  $\epsilon$.
 By inserting this into $
(-1)^r\frac{d^r}{dt^r}\big( \Upsilon^k(t)\big)^{*} \frac{d^s}{dt^s}
\Upsilon^k(t)$,  it can be seen that the dominant term is to be the
same whenever $r+s=2l+ 1$. Thus, it is clear that
\begin{equation*}
\begin{aligned}
&\Big[\big( \Upsilon^k \big)^*(\Upsilon^k)^{(2l+1)}
                                               -\big(\bar{\Upsilon^k}
                                               \big)^*
                                                (\Upsilon^k)^{(2l)}
+\cdots\pm  \big((\Upsilon^k)^{(l)} \big)^*
                                               (\Upsilon^k)^{(l+1)}  \Big]\\
=&(l+1)\big(\frac{\mathrm{i}k}{\epsilon}\dot{\phi} \big)^{2l+1}
 (\bar{\zeta}^k )^\intercal  \zeta^k
+\mathcal{O}\Big(\frac{1}{(l/M)!}\big(\frac{c}{\epsilon}\big)^{2l-2|k|}
\Big).
\end{aligned}
\end{equation*}
This implies  that the total derivative of \eqref{ld1} is given by
\begin{equation*}
\begin{aligned}
&\frac{-\mathrm{i}}{\epsilon} \sum\limits_{k}\frac{\mathrm{i}k}{h}
\sum\limits_{l\geq0}\Big[(-1)^l\beta_{2l} (l+1)\big(\frac{
kh}{\epsilon}\dot{\phi} \big)^{2l+1}
 (\bar{\zeta}^k )^\intercal  \zeta^k \Big]+\mathcal{O}(\epsilon)\\
=&\frac{1}{\epsilon h} \sum\limits_{k}\frac{k}{2}
\sum\limits_{l\geq0}\Big[(-1)^l\beta_{2l} (2l+2)\big(\frac{
kh}{\epsilon}\dot{\phi} \big)^{2l+1}
 (\bar{\zeta}^k )^\intercal  \zeta^k \Big]+\mathcal{O}(\epsilon)\\
=&  \frac{1}{\epsilon h} \sum\limits_{k}\frac{k}{2} 2
\tan\big(\frac{1}{2}\eta k
 \dot{\phi}\big) \sec^2\big(\frac{1}{2}\eta k
 \dot{\phi}\big)\abs{\zeta^k}^2 +\mathcal{O}(\epsilon)\\
 =&  \frac{2}{\epsilon h} \frac{1}{2} 2
\tan\big(\frac{1}{2}\eta
 \dot{\phi}\big) \sec^2\big(\frac{1}{2}\eta
 \dot{\phi}\big)\abs{\zeta^1}^2 +\mathcal{O}(\epsilon)\\
=&  \frac{\abs{\zeta^1}^2}{\epsilon^2}
\frac{\abs{B(\zeta^0)}}{\cos^2\big(\frac{1}{2}\eta
 \dot{\phi}\big)} +\mathcal{O}(\epsilon)=
\frac{1}{\cos^2\big(\frac{1}{2}\eta
 \dot{\phi}\big)}I(x_{n+\frac{1}{2}},v_{n+\frac{1}{2}})
 +\mathcal{O}(\epsilon).
\end{aligned}
\end{equation*}
Therefore, we obtain that
$$\mathcal{I}[\zeta](t_{n+\frac{1}{2}})=
\frac{1}{\cos^2\big(\frac{1}{2}\eta
 \dot{\phi}\big)}I(x_{n+\frac{1}{2}},v_{n+\frac{1}{2}})
 +\mathcal{O}(\epsilon)$$
and this completes the proof.

\end{proof}

By using the analysis presented in \cite{Hairer2018} and  Chap. XIII
of \cite{Hairer2006} again, Theorem \ref{nearly energy pre thm02s}
is proved.

\section{Concluding remarks} \label{sec:conclusions}
The study of long-term behaviour for an integrator when applied to
charged-particle dynamics is of great importance which have been
received a great attention. In this paper, we first presented two
two-step symmetric methods for charged-particle dynamics and then
analysed their long time behaviour not only in a normal magnetic
field but also in a strong magnetic field. Backward error analysis
and modulated Fourier expansion were used to prove the results.  By
the analysis of this paper, the long time conservations of    each
method were shown clearly for two cases of magnetic fields. In
compared with the Boris method and the variational method researched
recently in \cite{Hairer2018}, these two methods show some
superiorities in the conservations over long times.

\section*{Acknowledgement}
The first author is grateful to Professor Christian Lubich for the
discussions and helpful comments on the long term analysis of
methods for charged-particle dynamics.

\end{document}